\documentclass[preprint,10pt]{elsarticle}

\usepackage[utf8]{inputenc}
\usepackage{natbib}
\usepackage[english]{babel}

\usepackage{graphicx}
\usepackage{amssymb}
\usepackage{lineno}
\usepackage{amsthm}
\usepackage{marvosym}
\usepackage{amsmath}
\usepackage{color}
\usepackage{csquotes}

\usepackage{algorithmic}
\usepackage{algorithm}

\newcommand{\algorithmiccontinue}{\textbf{continue}}
\newcommand{\CONTINUE}{\STATE \algorithmiccontinue}

\newtheorem{definition}{Definition}[section]

\newtheorem{theorem}{Theorem}[section]
\newtheorem{remark}{Remark}[section]
\newtheorem{corollary}{Corollary}[section]
\newtheorem{lemma}{Lemma}[section]

\journal{Discrete Mathematics and Theoretical Computer Science}

\begin{document}
	\begin{frontmatter}

\title{ Transformations of Rectangular Dualizable Graphs }

		\author{Vinod Kumar\corref{cor1}\fnref{label}}
		 \ead{vinodchahar04@gmail.com}
		 \author{	Krishnendra Shekhawat\fnref{label}}
		 \ead{krishnendra.shekhawat@pilani.bits-pilani.ac.in}
	 \fntext[label]{Department of Mathematics, BITS Pilani, Pilani Campus, Rajasthan-333031, India }
		 \cortext[cor1]{Corresponding Author}


\begin{abstract}
A  plane graph is said to be a {\it rectangular graph } 
if  each of its edges  can be oriented  horizontal or vertical,  its internal regions are four-sided and it has a rectangular enclosure.
If  dual  of a planar graph is a rectangular graph, then the graph  is said to be a rectangular dualizable graph (RDG).  
In this paper, we present adjacency transformations between RDGs and present polynomial time algorithms for their transformations.

An RDG $\mathcal{G}=(V, E)$ is called maximal RDG (MRDG) 
if there does not exist an RDG $\mathcal{G'}=(V, E')$ with $E' \supset E$. An RDG $\mathcal{G}=(V, E)$ is said to be an edge-reducible if there exists an RDG $\mathcal{G'}=(V, E')$ such that $E\supset E'$. If an RDG is not edge-reducible, it is said to be an edge-irreducible RDG.
 We  show that there always exists an MRDG for a given RDG.  We also show that an MRDG is edge-reducible and can always be transformed to a minimal one (an edge-irreducible RDG). 
\end{abstract}

\begin{keyword}
 planar graph \sep rectangular floorplan \sep rectangular dualizable graph \sep VLSI circuit 
	
\end{keyword} 
			
\end{frontmatter}

\section{Introduction} \label{sec1}
Due to advances in VLSI technology, interconnection optimization is of  the major concern \cite{Sham07}, i.e., it is sometimes required to improve the interconnections of an existing VLSI circuit.   The modification of interconnection  can be seen in two ways: interconnection among  modules and interconnection of  modules to the outside units. To deal with the first case, we need to increase/decrease the adjacency relations among the modules as much as possible and  the second requires to increase the length (the number of modules adjacent to the exterior)  of exterior of the layout. Architecturally, a legacy  rectangular floorplan (RFP) can be reconstructed to suit modern lifestyles by changing the adjacency  relations of its rooms/rectangles. In this paper, we address  the graph theoretic characterization  to deal with the issues of interconnections in floorplanning.
\par
For a better understanding of the further text, we first explain the geometric duality relation between planar graphs and rectangular floorplans (RFPs). 
A graph $\mathcal{H}$ is called dual graph of a plane graph $\mathcal{G}$  if there is one to one corresponding between the vertices of   $\mathcal{G}$ and the regions of $\mathcal{H}$ and whenever any two vertices of $\mathcal{G}$ are adjacent, the corresponding regions of $\mathcal{H}$ are adjacent. A plane graph $\mathcal{G}$ is called rectangular  graph if  each of its edges  can be oriented  horizontal or vertical,  its internal regions are four-sided and it has a rectangular enclosure.  A graph admitting a rectangular dual graph is said to be a rectangular dualizable graph (RDG). A rectangular floorplan (RFP) can be seen as a particular embedding of the dual of an RDG  and is defined as a partition  of a  rectangle $\mathcal{R}$ into $n$-rectangles $R_1, R_2, \dots, R_n$ such that no four of them meet at a point. Fig. \ref{fig:f1} demonstrates the rectangular dualization method, i.e., the geometric duality relationship of  planar graphs and rectangular floor-plans. Consider the graph in Fig. \ref{fig:f1}a, whose   extended graph is shown in Fig \ref{fig:f1}b. It is dualized in Fig. \ref{fig:f1}c where a region $R_i$ corresponds to a vertex $v_i$. Further, on orienting each side of all regions of the dualized graph, horizontally or vertically, we obtain an RFP as shown in Fig \ref{fig:f1}d. 
\par
The point where three or more rectangles of a given RFP meet is called a joint.  We know that an RFP has 3-joints and 4-joints only where 4-joints are regarded as a limiting case of 3-joints \cite{rinsma1988existence}. Hence, abiding by the common design practices, we restrict ourselves to 3-joints only. Therefore throughout the paper, we transform an RFP with 3-joints to another RFP with 3-joints. The dual graph of such an RFP is always plane triangulated graph. 
\begin{figure}[H]
	\centering
	\includegraphics[width=0.95\linewidth]{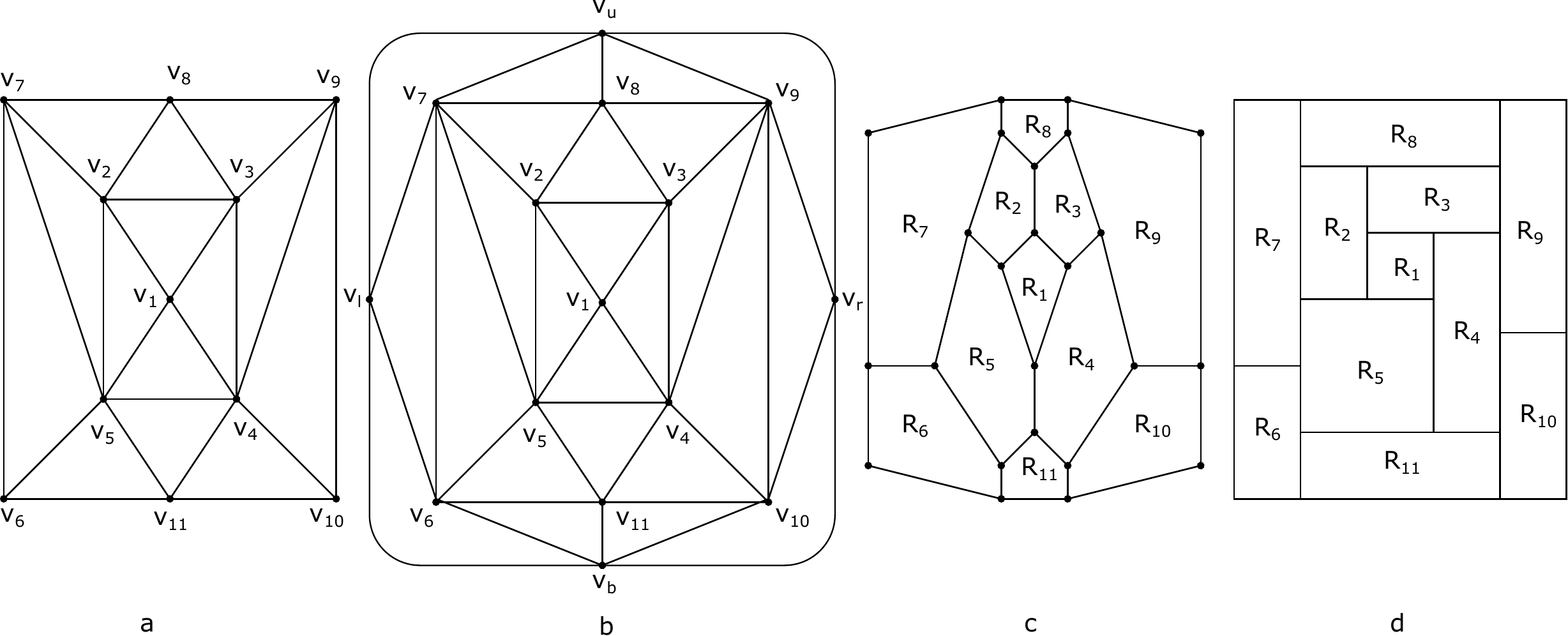}
	\caption{Rectangular dualization: (a) plane graph, (b) extended plane graph, (c)  rectangular dual graph and (d) rectangular floorplan }
	\label{fig:f1}
\end{figure}
\subsection{Related work}
The theory of RDGs is very restrictive \cite{Kozminski85,Bhasker87,Lai90}.  
 Constructive RFP algorithms \cite{Bhasker88,He93,Kant97,Eppstein12} based on this theory emphasizes on the packing of rectangles into a rectangular area.  Due to recent developments in the technology (VLSI design and architectural floorplanning), sometimes it is required to transform an existing RFP to  another RFP, if possible. Here, the idea is to recursively  improve an existing RFP until an optimal solution based on the user requirements, is attained.
 \par
 From graph theoretic context, RFP transformation techniques have not been much emphasized. It has  been addressed  by authors \cite{Lai88,Kozminski88,Tang90} where a topologically distinct RFP was obtained from an existing one for a given graph while preserving the adjacency relations of the existing one. Wang  \textit {et al.} \cite{wang2018customization} modified an existing RFP by inserting or removing a rectangle to it. Contrary to this,  in this paper, we address  RFP transformation technique from a graph theoretic context which is  based on changing  adjacency relations among rectangles. For any two given adjacent rectangles of an RFP,  it is interesting to identify whether they can be made nonadjacent  in such a way that  adjacency relations of  remaining rectangles of the RFP  remains  sustained and the resultant floor-plan is an RFP.  Conversely, can any two given non-adjacent rectangles of an RFP be  made adjacent in such a way that adjacency relations of  the remaining rectangles of the RFP  remains preserved and the resultant floor-plan is an RFP? 
 \par
 Enumerating of all RFPs composed of $n$-rectangles  has been remained a major issue in combinatorics  \cite{Yao03,Ackerman06,Reading12,Dawei14,Yamanaka17,Shekhawat18}.  These methods are not preferable because they produce  large solution space and therefore, it is computationally hard to   pick an RFP  suiting adjacency requirements of rectangles from such a large solution space.
\par 
In this paper, we present transformations between RFPs from graph theoretic context. However with the renewed interest in floorplannings, orthogonal floorplans from the graph theoretic context has also been well studied \cite{chang2015constrained,Alam13,Yeap93,He99,de2009rectilinear}.

\subsection{\rm{Our Results}}
Our main contribution is in two folds: (1). For a given  RDG, we aim to construct a new RDG by introducing new adjacency relations while preserving all the existing adjacency relations until no more adjacency relation can be added. (2). For a given  RDG, we aim to construct a new RDG by deleting adjacency relations while preserving all the other adjacency relations until no more adjacency relation can be removed. We first prove that these transformations are always possible and then present quadratic time algorithms for them. 

 The class of maximal RDGs (MRDGs) can play an important role in floorplanning because they  are rich in adjacency relations. Therefore, it is interesting to construct an MRDG of a given RDG.  Intuitively speaking, an MRDG is an RDG having maximal adjacency relations among its vertices. In fact, an MRDG with $n$ vertices has $2n-2$ or $3n-7$ edges. We show that  there always exists an MRDG for a given RDG.  Then we present a polynomial time algorithm that constructs the MRDG for the given RDG by adding new edges among its nonadjacent vertices. The number of such new edges is $2n-2-k$ or $3n-7-k$ where $k$ is the number of edges in the RDG. It would be more interesting if the given RDG is a path graph because then it requires the new edges to be added in bulk. Equivalently the new edges in bulk can be added to a Hamiltonian path of the given RDG to realize a new desired form of  the RDG. We also show that  a given MRDG is always edge-reducible and  can be reduced to an RDG and present an algorithm  that deletes the  edges of a given RDG in until it  is an RDG.
 
 In this way, we can just add or remove those edges of the RDG that are missing from one of the RDG and we are done.

 A brief description of the rest of the paper is as follows. In this article, we first survey the existing facts about RDGs in  Section \ref{sec2}. In Section \ref{sec3}, we introduce  MRDGs and edge-reducible RDGs. Then we show that an MRDG has $2n-2$ or $3n-7$ edges. MRDGs with  $2n-2$ edges are wheel graphs whereas MRDGs with  $3n-7$ edges are obtained from  the class of maximal plane graphs with the property that they do not have any separating triangles in their interiors by deleting one of their exterior edges.  In Section  \ref{sec4}, we prove  that it is always possible to construct an MRDG for a given RDG and present an algorithm for its construction from the given RDG. In Section \ref{sec5}, we show that it is always possible to transform an edge-reducible RDG to another RDG and present an algorithm for its reduction to a minimal one. Finally, we conclude our contributions and discuss future task in Section \ref{sec6}.

  \par	
  A list of  notations  used in this paper can be seen in  Table 1. 
  \begin{table}[H]
  	\centering
  	\begin{tabular}{|p{2cm}|p{9.6cm}|}
  		\hline
  		Symbol & Description \\
  		\hline
  		RFP & rectangular floorplan \\
  		\hline
  		RDG & rectangular dualizable graph \\
  		\hline
  		MRDG & maximal rectangular dualizable graph \\
  		\hline
  		$v_i$ &$i^{\text{th}}$ vertex of a graph  \\
  		\hline
  		$d(v_i)$ & degree of $v_i$   \\
  		\hline
  		$(v_i,v_j)$ & an edge incident to vertices $v_i$ and $v_j$\\
  		\hline
  		3-cycle & a cycle of length 3  \\
  		\hline
  		$v_iv_jv_k$ & a cycle passing through vertices $v_i$, $v_j$ and $v_k$ of an RDG\\
  		\hline
  		$R_i$ & a rectangle  corresponding $v_i$\\
  		\hline
  		$E_i$ & edge set of graph  $\mathcal{G}_i$ \\
  		\hline
  		$|E|$ & cardinality of a set $E$ \\
  		\hline
  		CIP(s) & corner implying path(s) \\
  		\hline
  	\end{tabular} 
  	\label{tab1}
  	\caption{\rm List of Notations}
  \end{table}

 \section{ Preliminaries} \label{sec2}
In this section, we survey some existing facts about RDGs which would be helpful in the coming sections.   

A planar graph  is a graph that can be  embedded in the plane without crossings. A plane graph is a planar graph with a fixed planar embedding. It partitions the plane
into connected regions called faces; the unbounded region is the exterior face (the outermost face) and all
other faces are interior faces. The vertices lying on the exterior face are exterior vertices and all other vertices are interior vertices. A planar (plane) graph is maximal if no new edge can be added to it without disturbing planarity. Thus each face of a maximal plane graph is a triangle  (including exterior face). If a connected graph has a cut vertex, then it is called a separable  graph, otherwise it is called a nonseparable  graph. 
Since floor-plans are concerned with connectivity, we only consider nonseparable (biconnected) and separable connected graphs in this paper. 
\begin{definition} \label{def21}
{\rm A  graph is said to be {\it rectangular graph } if each of its edges can be oriented horizontally or vertically such that it encloses a rectangular area. If the dual graph of  a planar graph is a rectangular graph, then the graph  is said to be a {\it rectangular dualizable graph} (RDG). In other words, a planar graph is rectangular dualizable (RDG) if its dual can be realized as a rectangular floorplan (RFP). An RFP is a partition  of a  rectangle $\mathcal{R}$ into $n$ rectangles $R_1, R_2,\dots, R_n$ provided that no four of them meet at a point.}
	
\end{definition}

Theorems \ref{thm21} and \ref{thm22} talks about the existence of an RDG, where the following definitions are required.

\begin{definition} \label{def22}
	{\rm \cite{Bhasker88} A separating triangle is a cycle of length 3 in a plane graph  that encloses atleast one vertex inside as well as outside. It is also known as a complex  triangle.}
\end{definition} 
Consider the graph shown in Fig. \ref{fig:f2}b having a cycle $\mathcal{C}$ of length 3  passing through the  vertices $v_1$, $v_2$ and $v_9$, enclosing the vertices $v_4$ and $v_5$, and having many vertices  outside $\mathcal{C}$. Therefore $\mathcal{C}$ is a   separating triangle.

\begin{definition} \label{def23}
	{ \rm \cite{Bhasker88} The block neighborhood graph (BNG) of a planar graph $\mathcal{G}$ is a graph in which each component of $\mathcal{G}$ is represented by a vertex and there is an edge between two vertices of the BNG  if and only if the two  corresponding components have a vertex in common.}
\end{definition} 
\begin{definition}\label{def24}
	{\rm \cite{Kozminski88} A shortcut in a biconnected plane graph $\mathcal{G}$ is an edge  incident to two vertices on the outermost cycle  of $\mathcal{G}$ and it is not a part  of this cycle.  A corner implying path (CIP) in $\mathcal{G}$ is a  $v_1-v_k$ path on the outermost cycle of $\mathcal{G}$ such that it does not contain any vertices of a shortcut other than $v_1$ and $v_k$ and the shortcut $(v_1,v_k)$ is called a critical shortcut. A critical CIP in a biconnected component $H_k$ of a separable plane  graph $\mathcal{G}_1$ is a CIP of $H_k$ that does not contain cut vertices of $\mathcal{G}_1$ in its interior.}
\end{definition} 
For  a better understanding of Definition \ref{def24}, consider the graph shown in Fig. \ref{fig:f2}a. 
\begin{itemize}
    \item Edges $(v_1,v_3)$, $(v_6,v_8)$ and $(v_4,v_9)$ are shortcuts,
    \item $v_1v_2v_3$ and $v_6v_7v_8$ are CIPs,
    \item $v_9v_1v_2v_3v_4$ is not a CIP because it contains the endpoints of the shortcut $(v_1,v_3)$ and hence $(v_9,v_4)$ is not a critical shortcut (CIPs may have the same endpoints, but they are edge disjoint).
\end{itemize}

\begin{figure}[H]
	\centering
	\includegraphics[width=0.80\linewidth]{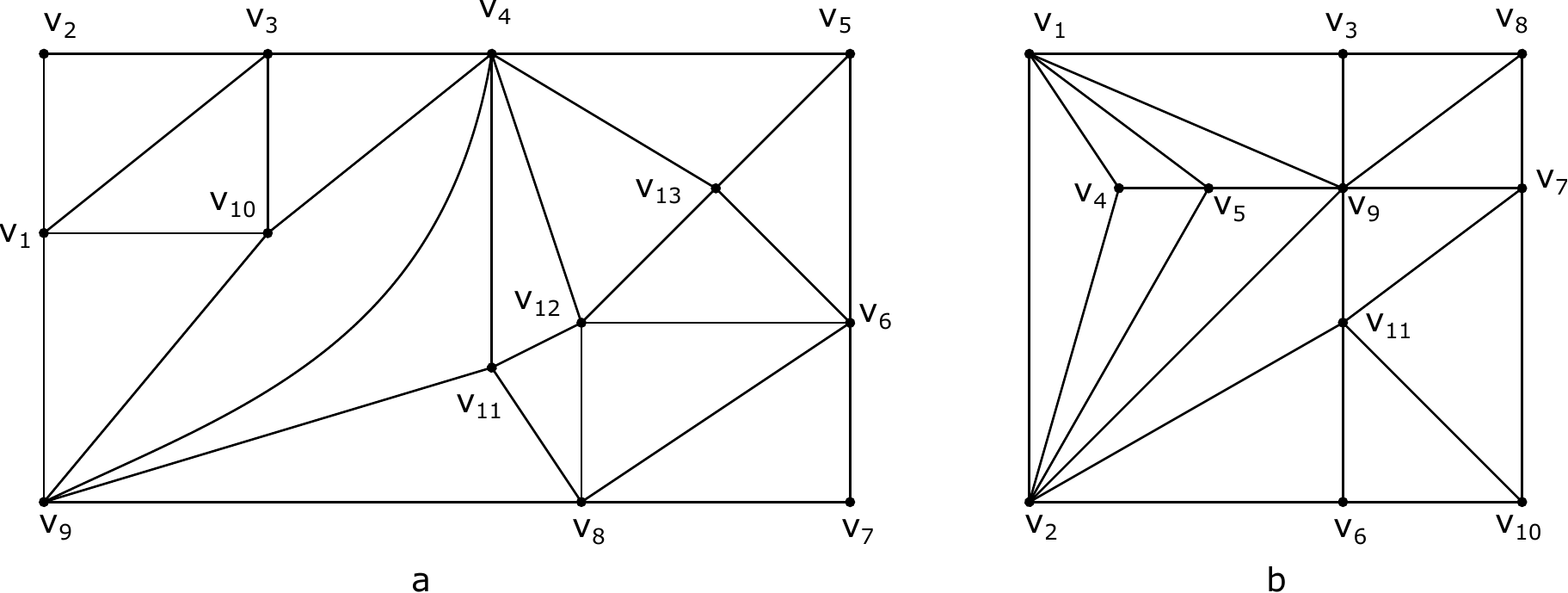}
	\caption{(a) Edges $(v_1,v_3)$, $(v_6,v_8)$ and $(v_4,v_9)$ are shortcuts. $v_1v_2v_3$ and $v_6v_7v_8$ are CIPs and (b) $\triangle v_1v_2v_9$ is a separating triangle }
	\label{fig:f2}
\end{figure}

\begin{theorem} \label{thm21}
	{\rm \cite[Theorem 3]{Kozminski85}  A nonseparable  plane graph $\mathcal{G}$ with triangular interior faces (regions) is an RDG if and  only if  it has at most 4 CIPs and has no separating triangle.}
\end{theorem}

The graph shown in Fig. \ref{fig:f2}a is an RDG while the graph in Fig. \ref{fig:f2}b is not an RDG because of the presence of a separating triangle.

\begin{theorem} \label{thm22}
{\rm\cite[Theorem 5]{Kozminski85} A separable connected plane graph $\mathcal{G}$ with triangular interior  faces (regions)  is an RDG if and only if
\begin{enumerate}
\item[i.] $\mathcal{G}$ has no separating triangle,
\item[ii.] BNG of $\mathcal{G}$ is a path,
\item[iii.]  each maximal block\footnote{A maximal block of a graph $\mathcal{G}$ is a maximal biconnected subgraph of $\mathcal{G}$.} corresponding to the endpoints of	the BNG contains at most 2 critical CIPs, 
\item[iv.] no other maximal block contains a critical CIP.	
	\end{enumerate}}
\end{theorem}

 \section{ Introduction to an MRDG and an edge-reducible RDG} \label{sec3}
In this section, we introduce an MRDG and  some  important properties of the MRDG. Further we also introduce an edge-reducible RDG.
  
  \begin{definition} \label{def31}
	{\rm An RDG $\mathcal{G}=(V, E)$ is called maximal RDG (MRDG) 
	 if there does not exist an RDG $\mathcal{G'}=(V, E')$ with $E' \supset E$.}
	 \end{definition} 
  
\begin{definition} \label{def32}
{\rm An RDG $\mathcal{G}=(V, E)$ is said to be  edge-reducible if there exists an RDG $\mathcal{G'}=(V, E')$ such that $E\supset E'$. If an RDG is not edge-reducible, it is said to be an edge-irreducible RDG.}
\end{definition} 

 For a better understanding of Definitions \ref{def31} and \ref{def32}, refer to Fig. \ref{fig:f3}. In Fig. \ref{fig:f3}a, the given RDG is an MRDG since adding a new edge to it, its exterior becomes triangular and hence one of the  exterior modules in its floorplan will be non-rectangular. The corresponding RFP of the MRDG is shown in  \ref{fig:f3}b where $R_i$ corresponds to $v_i$. In Fig. \ref{fig:f3}c, the given graph $\mathcal{G}$ is an  edge-irreducible biconnected RDG since after removing an edge from it, it no longer remains an RDG. In fact, on removing an exterior edge from $\mathcal{G}$, it  transforms to a separable connected graph where one of its blocks contains three critical CIPs. It is noted that we can not remove an interior edge from $\mathcal{G}$ since  the resultant graph must have triangular interior regions. This is because we only consider plane graphs with triangular interior regions.
 
 \begin{theorem}
{\rm The number of edges  in an MRDG is $2n-2$ or $3n-7$ where $n$ denotes the number of vertices in the MRDG.}
\end{theorem}
\begin{proof}
 Let $\mathcal{M}$ be an MRDG with $n$ vertices. If $d(v_i)=n-1$ for some vertex $ v_i \in \mathcal{M}$, then clearly it is a wheel graph $W_n$\footnote{A wheel graph $W_n$ is a graph in which a single vertex is adjacent to $n-1$ vertices lying on a cycle.}. $W_n$ is independent of  separating triangle and CIP. By Theorem \ref{thm21}, it is an RDG. Note that adding a new edge to $W_n$ creates a separating triangle passing through its central vertex and two of its exterior vertices. This implies that  it is maximal RDG. Now by degree sum formula, the sum of degree of all vertices of a graph is twice of the number of its edges. This implies that $3(n-1)+(n-1)=2$(the number of edges in $W_n$) and hence  the number of edges in $W_n$ is $2n-2$.
 
 We have shown that $W_n$ is an MRDG with $2n-2$ edges and  in this case, the number of edges  in $\mathcal{M}$ is $2n-2$.
	   
Now suppose that $d(v_i)<n-1, ~\forall ~ v_i~ \in \mathcal{M}$. Consider a maximal plane  graph $\mathcal{G}$ with $n$ vertices such that there does not exist any separating triangle in its interior.  We claim that $\mathcal{M}=\mathcal{G}-(v_i,v_j)$ for some exterior edge $(v_i,v_j)$ of $\mathcal{G}$. In order to claim  this, we  need to show that  $\mathcal{G}-(v_i,v_j)$ is an MRDG with $3n-7$ edges.

By our assumption on $\mathcal{G}$, it is evident that $\mathcal{G}-(v_i,v_j)$  has no separating triangle. 
 Suppose that $\mathcal{G}-(v_i,v_j)$ has a CIP.  Then there is a shortcut $(v_s,v_t)$ in $\mathcal{G}-(v_i,v_j)$. This implies that $\mathcal{G}$ has a separating triangle $v_sv_tv_e$  in its interior where $v_e$  is  its exterior vertex.  This is a contradiction to our assumption that $\mathcal{G}$ has no separating triangle in its interior. Thus we see that $\mathcal{G}-(v_i,v_j)$ has no CIP. Thus By Theorem \ref{thm21}, it is an RDG.
 
 The number of edges in a maximal plane graph is $3n-6$. This implies that $\mathcal{G}-(v_i,v_j)$  has $3n-7$ edges. Note that adding a new edge to $\mathcal{G}-(v_i,v_j)$, creates a separating triangle in  $\mathcal{G}-(v_i,v_j)$ and hence it is an MRDG.

Thus we see that  $\mathcal{G}-(v_i,v_j)$ is an RDG with $3n-7$ edges and hence  $\mathcal{M}$ is an MRDG with $3n-7$ edges. \\

\end{proof}
\begin{theorem}
{\rm The number of  vertices on the outermost cycle of an MRDG with $n$ vertices is $4$ or $n-1$.}
\end{theorem}
\begin{proof} 
Let $\mathcal{M}$ be an MRDG with $n$ vertices. If $\mathcal{M}$ is  a wheel graph $W_n$, then it is clear that it has $n-1$ vertices on its exterior. Otherwise it is obtained  from a maximal plane graph that has no separating triangle in its interior by deleting one of its exterior edges. But a maximal plane graph has 3 vertices on its exterior. Then it is evident that  $\mathcal{M}$ has 4 vertices on its exterior.
\end{proof}
 
 \begin{figure}[H]
	\centering
	\includegraphics[width=0.90\linewidth]{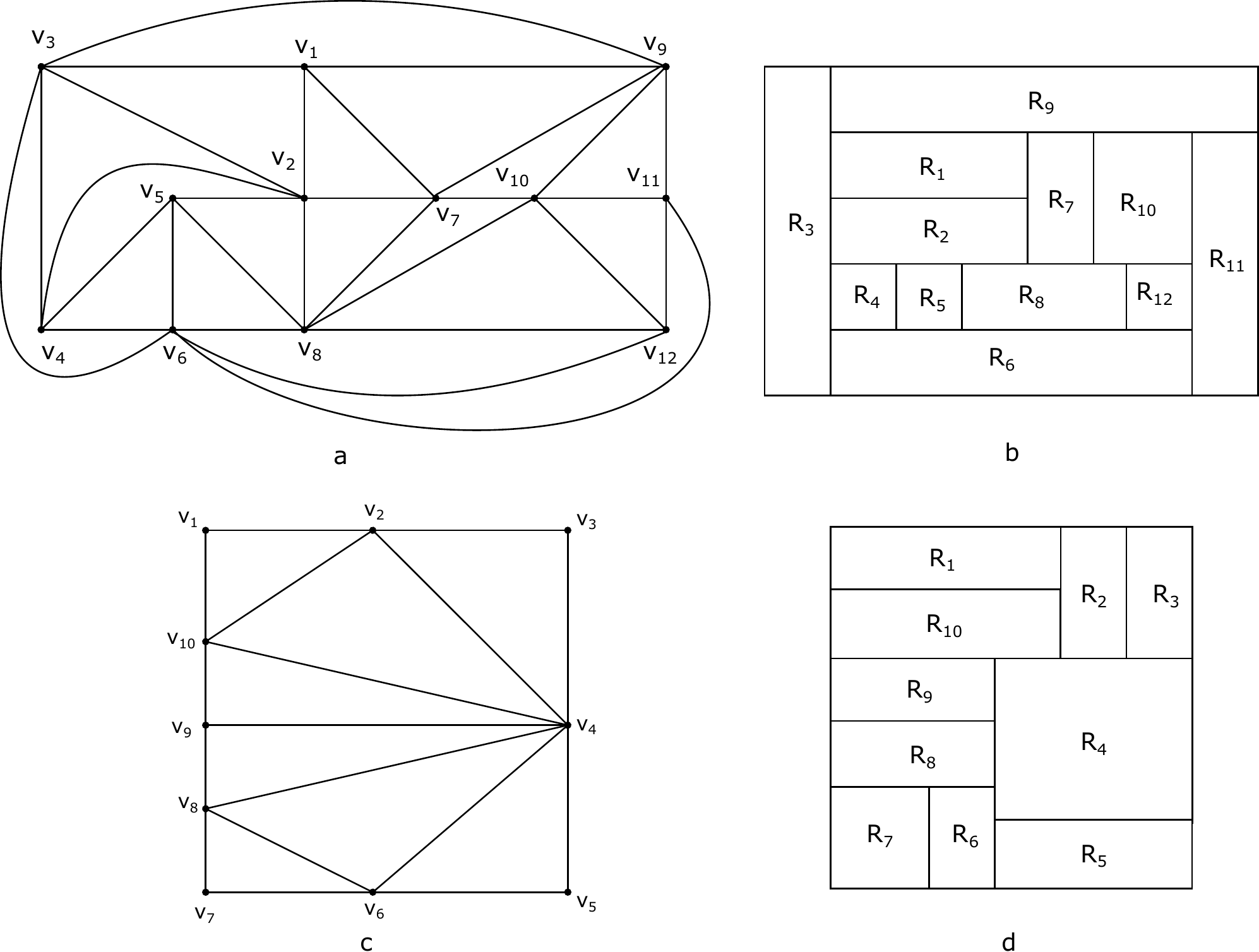}
	\caption{(a-b) An MRDG and its RFP, and (c-d) an edge-irreducible biconnected RDG and  its RFP.}
	\label{fig:f3}
\end{figure}

\section{MRDG Construction} \label{sec4}
  In this section,  we first prove that it is always possible to construct an MRDG $\mathcal{M} =(V, E)$ for a given RDG $\mathcal{G}=(V, E_1)$  such that $E_1\subset E$. Then we present an algorithm for its construction corresponding to $\mathcal{G}$. 
  
   In order to prove the main result, we first need to prove some lemmas. Denote $|N(v_i) \cap N(v_j)|$ by $s$ for any two adjacent  vertices $v_i$ and $v_j$ of an RDG $\mathcal{G}$.

\begin{lemma} \label{lem41}
{\rm If $s= 0$ , then  $(v_i,v_j)$ is an exterior edge of $\mathcal{G}$.}
\end{lemma}
\begin{proof}
 For $s=0$, $(v_i,v_j)$ is a cut-edge (bridge) and hence is an exterior edge.
 \end{proof}
 
 \begin{lemma} \label{lem42}
	{\rm For all adjacent  vertices $v_i$ and $v_j$, $s \leq 2$.}
\end{lemma}
\begin{proof}
   To the contrary, suppose that  $s=3$. 
   Consider a plane embedding $\mathcal{G}_e$ of $\mathcal{G}$. Since, $s=3$, $N(v_i)$ and $N(v_j)$ must have 3 common vertices $v_k$, $v_l$ and $v_m$ which results in 3 cycles $v_iv_jv_k$, $v_iv_jv_l$ and $v_iv_jv_m$ in $\mathcal{G}_e$. Now $(v_i, v_j)$ is a common edge in these 3 cycles. This implies that  atleast 2 of 3 cycles would lie on the same side of $(v_i, v_j)$ in  $\mathcal{G}_e$. This means that  one of the cycles encloses some vertex $v_t$  of the other cycle and hence is not a face in $\mathcal{G}_e$. Therefore its removal  results $\mathcal{G}_e$ in a disconnected graph and hence it is a separating triangle in $\mathcal{G}$, which is a contradiction to Theorems \ref{thm21} and \ref{thm22} since $\mathcal{G}$ is an RDG. Similarly, if $s\geq 3$, we arrive at the contradiction. 
\end{proof}

\begin{lemma} \label{lem43}
{\rm $(v_i,v_j)$ is an interior edge of $\mathcal{G}$ if and only if  $s = 2$.}
\end{lemma}
\begin{proof}
First suppose that  $s = 2$. We need to show that $(v_i,v_j)$ is an interior edge in $\mathcal{G}$. To the contrary, suppose that  $(v_i,v_j)$ is an  exterior edge  of  $\mathcal{G}$. Let $N(v_i) \cap N(v_j)= \{v_k,v_l\}$. Since $(v_i,v_j)$ is an exterior edge, there exist two triangles $v_iv_jv_k$, $v_iv_jv_l$ in the plane embedding of  $\mathcal{G}$ such that both lie on the same side of $(v_i,v_j)$. This implies that one of them contains the other and hence is not a region (face) and is a separating triangle. This is a contradiction to Theorems \ref{thm21} and \ref{thm22} since $\mathcal{G}$ is an RDG.
\par
Conversely, suppose that $(v_i,v_j)$ is an interior edge in $\mathcal{G}$. Since $\mathcal{G}$ is an RDG, each of its interior regions is triangular. This implies that there exist two triangles $v_iv_jv_k$ and $v_iv_jv_l$ in the plane embedding of  $\mathcal{G}$. Hence  $N(v_i)$  and $N(v_j)$ have atleast two vertices in common, i.e., $s \geq 2$. By Lemma \ref{lem42}, we have $s \leq 2$. Hence $s =2$.
\end{proof}

\begin{corollary} \label{cor41}
{\rm  If $s= 1$, then  $(v_i,v_j)$ is an exterior edge of $\mathcal{G}$.}
\end{corollary}

\begin{proof}
It is the direct consequence of Lemmas \ref{lem41} and \ref{lem43}.
\end{proof}

\begin{lemma} \label{lem44}
{\rm  It is always possible to construct a nonseparable (biconnected) RDG from  a separable connected   RDG by adding edges to it.}
\end{lemma}
\begin{proof}
Let $\mathcal{G}_1 =(V,E_1)$ be a separable connected RDG such that it has atleast one  bridge (cut-edge).  Suppose that  $L_1=\{ (v_i,v_j) \in E_1~|~~ |N(v_i) \cap N(v_j)|=0\}$ and $L_2=\{ (v_a,v_b) \in E_1~|~~ |N(v_a) \cap N(v_b)|=1\}$. Consider two  adjacent edges, $(v_i,v_j)$ from $L_1$ and  $(v_j,v_k)$ from $L_2$ such that $|N(v_i) \cap N(v_k)|=1$. Such selection is always possible since both edges belongs to different blocks and   $N(v_i) \cap N(v_k)=\{v_j\}$. 
\par
 Construct a graph  $\mathcal{G}_2=(V, E_2)$ where $E_2=E_1\cup \{(v_i,v_k)\}$. To prove $\mathcal{G}_2$ is an RDG, we prove the following:
 \begin{itemize}
     \item there does not exist  a separating triangle  passing through  $(v_i,v_k)$  in $\mathcal{G}_2$
     
     There would be a separating triangle passing through $(v_i,v_k)$ in $\mathcal{G}_2$ if $|N(v_i) \cap N(v_k)|=2$ in $\mathcal{G}_2$. In this case, 
     $v_j, v_r \in N(v_i) \cap N(v_k)$ such that $v_r$ lies inside the triangle passing through $(v_i,v_k)$.
     But $(v_i,v_j) \in L_1$ and $(v_j,v_k) \in L_2$. Therefore by Lemma \ref{lem41} and Corollary \ref{cor41}, both $(v_i,v_j)$ and $(v_j,v_k)$ are the exterior edges and $v_i$ and $v_k$ belongs to different blocks in  $\mathcal{G}_1$. Hence in $\mathcal{G}_2$,
      $|N(v_i) \cap N(v_k)|=1$ is the only possibility.  
     
     \item the number of critical CIPs in  $\mathcal{G}_2$ can not exceed the number of critical CIPs in $\mathcal{G}_1$ 
     
     In $\mathcal{G}_2$, a  critical CIP can only pass through $(v_i,v_k)$, which already  passes through $(v_i,v_j)$ and $(v_j,v_k)$ in $\mathcal{G}_1$. But $v_j$ is a cut vertex which is a contradiction to the fact that a critical CIP never passes through a cut vertex. 
 \end{itemize}

Since $\mathcal{G}_1$ is an RDG, each of its region is triangular. The new edge $(v_i,v_k)$ is added with the property that $|N(v_i) \cap N(v_k)|=1$. By Corollary \ref{lem41}, $(v_i,v_k)$ is exterior edge in $\mathcal{G}_2$.  Therefore, the new region  $v_iv_jv_k$ is triangular in $\mathcal{G}_2$. By Theorem \ref{thm22}, $\mathcal{G}_2$ is an RDG.  
\par
After adding $(v_i,v_k)$ to $\mathcal{G}_1$,  the edge $(v_i,v_j)$ from $L_1$ belongs to $L_2$ since  $|N(v_i) \cap N(v_j)|=1$ ($|N(v_i) \cap N(v_j)|=\{v_k\}$). Therefore a recursive process shows that at the iteration until $L_1$ is empty,  $\mathcal{G}_{k+1}=(V,E_{k+1})$ becomes a separable connected RDG with cut-vertices (vertex), but no cut edge where $E_{k+1}=E_k \cup (v_a,v_c)$ such that $(v_a,v_b)$ is from $L_1$ and  $(v_b,v_c)$ is from $L_2$ with the property $|N(v_a) \cap N(v_c)|=1$. In this way, we can construct a separable connected RDG having cut-vertices of  the given separable connected RDG only.
\par
It now remains to show that it is always possible to construct a nonseparable (biconnected) RDG  of  the given separable connected RDG $\mathcal{G}_1 =(V,E_1)$ having cut-vertices but no cut-edges. Let $v_t$ be its cut-vertex. Since it has no cut-edge, $d(v_t) \geq 4$. A plane embedding  of $\mathcal{G}_1$  with exterior cycles $C_1$ and $C_2$  sharing a cut vertex $v_t$ is  shown in Fig. \ref{fig:f4}a. It is evident from this embedding that there is no separating triangle passing though the new added edges (red edges) in the resultant graph  shown in Fig. \ref{fig:f4}b. Since $v_t$ is a cut-vertex,   none of the vertices $v_1$, $v_2$, $v_3$ and  $v_4$ in Fig. \ref{fig:f4}, which are adjacent to $v_t$, can  be the endpoints of a shortcut in  $\mathcal{G}_1$. This implies that the number of CIPs in the resultant graph (shown in Fig. \ref{fig:f4}b) can not exceed than the number of critical CIPs in $\mathcal{G}_1$. This proves the required result. 
\end{proof}

\begin{figure}[H]
	\centering
	\includegraphics[width=.95\linewidth]{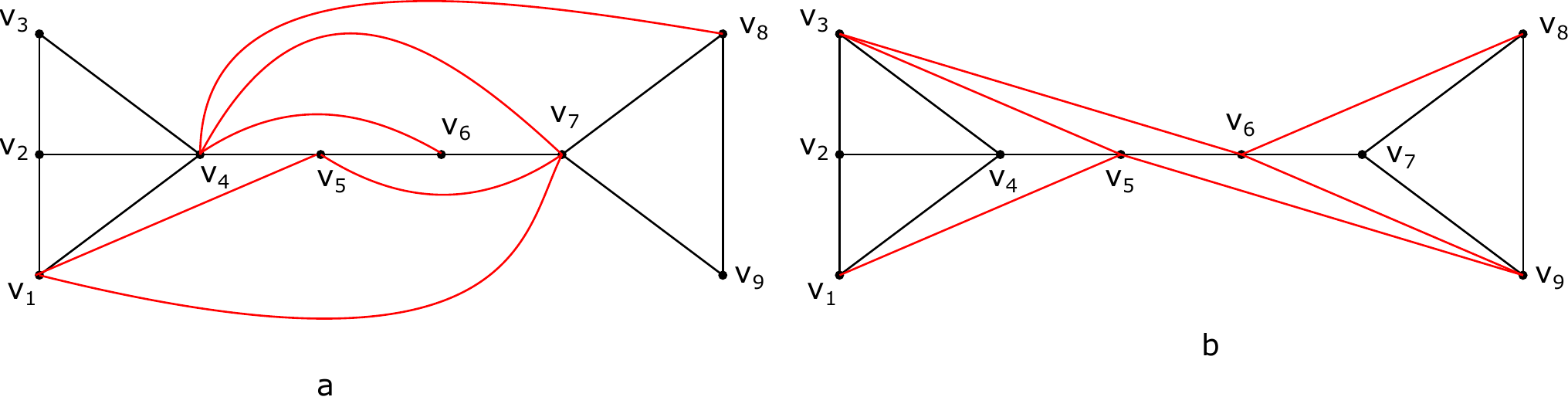}
	\caption{\rm (a) A random addition of new edges (red edges) to a separable connected RDG destroy the RDG property because of generating a separating triangle $v_4v_5v_7v$  in the resultant graph (b) while adding new edges (red edges) by Lemma \ref{lem44} do not destroy RDG property ( no separating triangle and any CIP in the resultant graph).}
	\label{fig:f8}
\end{figure}

 \begin{remark}
 {\rm It is not straight forward to add edges to a  separable connected RDG maintaining  RDG property.  Randomly adding new edges to an RDG can destroy the RDG property, i.e.,  can produce either a CIP or a separating triangle in the resultant graph.  In the lie of this, we have shown the  procedure of adding edges by Lemma \ref{lem44}. For instance, consider a separable connected RDG shown in Fig. \ref{fig:f8}a. It is transformed to a nonseparable graph by adding new edges (red edges) randomly. As a result, the nonseparable graph   thus obtained is not an RDG. In fact, it contains a separating triangle $v_4v_5v_7$.  On the other hand, red edges are added to the same graph by using Lemma \ref{lem44} in order to construct a nonseparable RDG shown in Fig. \ref{fig:f8}b. Thus Lemma \ref{lem44} is helpful in introducing pattern of new edges to be added to a  separable  connected  RDG to be a nonseparable RDG.}
 \end{remark}

\begin{figure}[H]
	\centering
	\includegraphics[width=.70\linewidth]{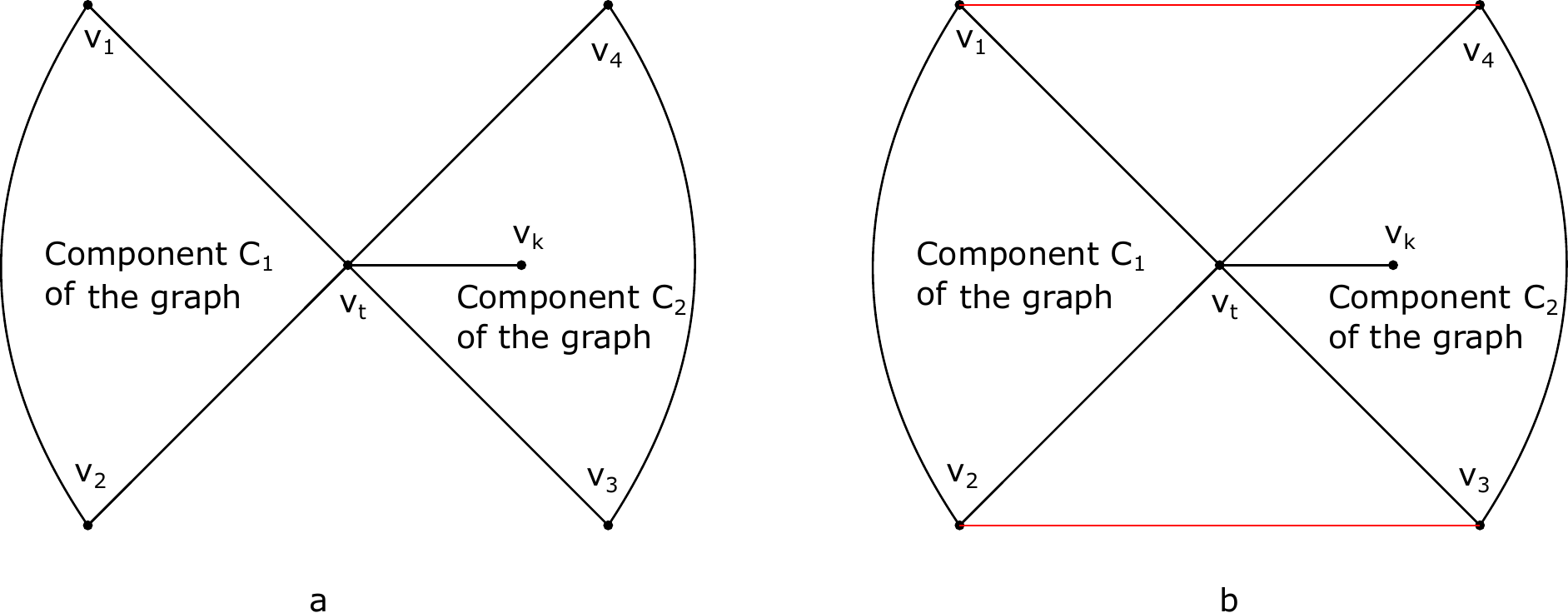}
	\caption{\rm Constructing a nonseparable RDG of  a separable connected RDG $\mathcal{G}_1$  with a cut-vertex $v_t$ shared by its two components $C_1$ and $C_2$.}
	\label{fig:f4}
\end{figure}

\begin{lemma} \label{lem45}
{\rm It is always possible to construct an MRDG from  a biconnected  RDG by adding edges to it.}
\end{lemma}
\begin{proof}
Let $\mathcal{G}_1 =(V,E_1)$ be a biconnected RDG. If $|E_1|=3|V|-7$ or $\mathcal{G}_1= W_n$, then $\mathcal{G}_1$ is itself an MRDG and the proof is obvious.  
\par
Suppose that $|E_1|<(3|V|-7)$ and $\mathcal{G}_1 \neq W_n$. Assume that $L_2=\{ (v_i,v_j) \in E_1 \mid |N(v_i) \cap N(v_j)|=1\}$. By Lemma \ref{lem41}, $L_2$ is a list of all exterior edges of $\mathcal{G}_1$.
\par
We now prove that  there exists atleast a pair of adjacent edges $(v_i,v_j)$ and  $(v_j,v_k)$ in $L_2$ such that $|N(v_i) \cap N(v_k)|=1$.  If such pair does not exist,  then $|N(v_a) \cap N(v_c)|=2$ for each pair  $(v_a,v_b)$ and $(v_b,v_c)$ in $L_2$. In fact, since $\mathcal{G}_1$ is a biconnected graph, by Lemma \ref{lem42}, we have $|N(v_a) \cap N(v_c)| \in \{1,2\}$. Let $v_1v_2 \dots v_p$ be the outermost cycle of $\mathcal{G}_1$. Note that all edges $(v_1, v_2)$, $(v_2, v_3)$ \dots $(v_{p-1}, v_p)$ and $(v_p, v_1)$ are exterior and hence by Lemma \ref{lem42}, all these edges belongs to  $L_2$. Now if we choose  $(v_1,v_2)$ and  $(v_2,v_3)$, then $|N(v_1) \cap N(v_3)|=\{v_2, v_c\}$. Again if we choose  $(v_2,v_3)$ and  $(v_3,v_4)$, then $|N(v_2) \cap N(v_4)|=\{v_3, v_c\}$. Continuing in this way, we see that all the exterior vertices are adjacent to $v_c$. Observe that the vertex $v_c$ and every   adjacent exterior vertices $v_i$ and $v_j$  forms a triangle. Therefore, if $\mathcal{G}_1$ has any other vertex (except $v_1, v_2, \dots, v_p$ and $v_c$ ), it would lie  inside the  triangle $v_iv_jv_c$, which is a separating triangle. This contradicts the fact that  $\mathcal{G}_1$ is an RDG. This implies that $\mathcal{G}_1$ cannot have any other vertex (except $v_1, v_2, \dots, v_p$ and $v_c$)
which concludes that $\mathcal{G}_1$ is a wheel graph $W_n$ which is again a contradiction since we assumed that $\mathcal{G}_1 \neq W_n$. This proves our claim. 
\par
Choose two  adjacent edges $(v_i,v_j)$ and  $(v_j,v_k)$ from $L_2$ such that $|N(v_i) \cap N(v_k)|=1$ and construct  a graph $\mathcal{G}_2=(V, E_2)$ where $E_2=E_1\cup \{(v_i,v_k)\}$.
\par
Now we show  that the number of CIPs in $\mathcal{G}_2$ can not exceed the number of CIPs in $\mathcal{G}_1$. For $\mathcal{G}_2$, there are the following possibilities:
\begin{enumerate}
	  \item[i.] None of vertices $v_i$ and $v_k$ is the endpoint of a shortcut in $\mathcal{G}_1$,
	 \item[ii.] One of vertices $v_i$ and $v_k$ is the endpoint of a shortcut  in $\mathcal{G}_1$,
	 \item[iii.] Both vertices $v_i$ and $v_k$ are the endpoints of a shortcut  in $\mathcal{G}_1$.
\end{enumerate} 
	These 3 possibilities are shown in Fig. \ref{fig:f5}a-\ref{fig:f5}c respectively. In the first case, clearly there is no CIP passing through  $(v_i,v_k)$ in $\mathcal{G}_2$.  In the second case, $v_iv_kv_{k+1}\dots v_{q-1}v_q$  becomes a CIP in $\mathcal{G}_2$ and $v_iv_jv_kv_{k+1}\dots v_{q-1}v_q$   no longer remains a CIP in $\mathcal{G}_2$. In fact, the edges $(v_i,v_j)$ and $(v_j,v_k)$ of the existing CIP in $\mathcal{G}_1$ are replaced by $(v_i,v_k)$ in  $\mathcal{G}_2$.  Thus, in this case, the number of CIPs do not get increased. The third case is not possible since $|N(v_i) \cap N(v_k)|=2$. In fact, $N(v_i) \cap N(v_k)=\{v_j, v_s\}$ and $\mathcal{G}_2$ is obtained from $\mathcal{G}_1$ by adding an edge $(v_i,v_k)$ such that  $|N(v_i) \cap N(v_k)|=1$. This proves our claim.
\par
Now we claim that  there does not exist  a separating triangle  passing through  $(v_i,v_k)$  in $\mathcal{G}_2$. Since $|N(v_i) \cap N(v_k)|=1$, i.e.,  $N(v_i) \cap N(v_k)=\{v_j\}$. Therefore $v_iv_jv_k$ is the only cycle of length 3 having no vertex inside and passing through $(v_i,v_k)$ in $\mathcal{G}_2$. This shows that $v_iv_jv_k$  is not a separating triangle, it is  a new added triangular region (face)  in $\mathcal{G}_2$. By Theorem \ref{thm21}, $\mathcal{G}_2$  is an RDG.   A recursive process  shows that each $\mathcal{G}_i=(V,E_i)$, $(i\geq 3)$ is an RDG where $E_i=E_{i-1} \cup \{(v_a,v_c)\}$ such that $|N(v_a) \cap N(v_c)|=1$ for some edges $(v_a,v_b), (v_b,v_c)$ belong to $L_2$ which is defined as $L_2=\{ (v_i,v_j) \in E_{i-1} \mid |N(v_i) \cap N(v_j)|=1\}$.
\par
It can be noted that the recursive process will terminate when the outermost cycle has four vertices for some RDG $\mathcal{G}_k$.  In fact,   $(v_i,v_j)$, $(v_j,v_k)$, $(v_k,v_l)$ and $(v_l,v_i)$ are four edges constituted by  the four exterior vertices $v_i$,  $v_j$,  $v_k$ and  $v_l$ of  some RDG $\mathcal{G}_k$. For any  two edges $(v_a,v_b)$ and $(v_b,v_c)$, we have  $|N(v_a) \cap N(v_c)|=2$. This terminate our process. On the other hand, there does not exist any other way for adding a new edge such that the resultant graph is an RDG with a new triangular region. Recall that a maximal plane graph has $3|V_1|-6$ edges where $V_1$ denotes its vertex set and has all triangular regions including exterior. In our case, every region of   $\mathcal{G}_k$ is triangular, but exterior is quadrangle. This implies that the number of edges in  $\mathcal{G}_k$ is  $3|V|-7$ and hence it is an MRDG. This completes the proof of lemma.
\end{proof}

 From Lemmas \ref{lem44}-\ref{lem45}, we conclude that the following main result of the paper.
\begin{theorem} \label{thm41}
{\rm  It is always possible to construct an MRDG of  a given RDG.}
\end{theorem}

\begin{algorithm}[H]
	\caption{\textbf{Constructing  an MRDG of a given RDG }}
	\begin{algorithmic}[1]
		\REQUIRE  An RDG $\mathcal{G}_1=(V,E_1)$ 
		\ENSURE An MRDG $\mathcal{M} =(V, E)$ for  $\mathcal{G}=(V, E_1)$  such that $|E_1|<|E|$   	
		
		\STATE $L_{1} \leftarrow \phi$	
		\STATE $L_{2} \leftarrow \phi$
		
		\FORALL {$(v_i,v_j)\in E_1$}
		\STATE $s \leftarrow |N(v_i) \cap N(v_j)|$
		\IF {$s==0$}
		\STATE $L_1 \leftarrow L_1 \cup \{(v_i,v_j)\}$
		\ELSIF {$s==1$}	
		\STATE $L_2 \leftarrow L_2 \cup \{(v_i,v_j)\}$
		\ELSE
		\CONTINUE
		\ENDIF
		\ENDFOR
		
		\FORALL {$(v_i,v_j) \in L_1$}
		\IF {$(v_j,v_k) \in L_2$}
		\STATE $L_2 \leftarrow (L_2 \cup \{(v_i,v_j), (v_i,v_k)\})-\{(v_j,v_k)\}$
		\STATE $ E_1\leftarrow E_1 \cup \{(v_i,v_k)\}$
		\ELSE
		\CONTINUE
		\ENDIF
		\ENDFOR
		
		\FORALL {$(v_i,v_j), (v_j,v_k) \in L_2$} 
		\IF {$ |N(v_i) \cap N(v_k)|==1$}
		\STATE $L_2 \leftarrow L_2 \cup \{(v_i,v_k)\}-\{(v_i,v_j), (v_j,v_k)\}$
		\STATE $ E_1\leftarrow E_1 \cup \{(v_i,v_k)\}$
		\ELSE
		\CONTINUE
		 \ENDIF
		\ENDFOR
		
		\FORALL {$(v_i,v_j) \in E_1$}
		\IF {$(v_i,v_j) \in (E_1-\{(v_i,v_j)\})$}
		\STATE $ E_1\leftarrow E_1 -\{(v_i,v_j)\}$
		\ELSE
		\CONTINUE
		\ENDIF
		\ENDFOR
		\RETURN $\mathcal{G}$
	\end{algorithmic} \label{algo41}
\end{algorithm}
Since the output of Algorithm \ref{algo41} is an MRDG having four vertices on its exterior, the corresponding RFP would have four rectangles on the exterior. It may not always be desirable to transform a given RDG to an  MRDG. In such a case, we can replace $L_2$ by $L_2-A$ where $A$ is the set of edges not to be added to the given RDG. Thus we can obtain the required RDG from a given RDG.

 For a better understanding to Algorithm  \ref{algo41}, we explain its steps through  an example. Consider an RDG $\mathcal{G}_1=(V,E_1)$ shown in Fig. \ref{fig:f6}a. First of all,  Algorithm \ref{algo41}  computes two sets $L_1$ and $L_2$  (the lines $3-12$) from $\mathcal{G}_2$ such that  $L_1$  contains those edges whose endpoints have no common  vertex and $L_2$ contains  those edges whose endpoints have exactly one common  vertex. Then  $L_1=\{(v_7,v_{10})\}$ and $L_2=\{(v_1,v_2)$, $(v_2,v_3),$ $(v_3,v_4),$ $(v_4,v_6),$ $(v_6,v_8),$ $(v_7,v_8),$ $(v_{10},v_{12}),$ $(v_{11},v_{12},(v_9,v_{11}),$ $(v_9,v_{10}),$ $(v_7,v_1) \}$.
 \par  
 \begin{figure}[H]
	\centering
	\includegraphics[width=.85\linewidth]{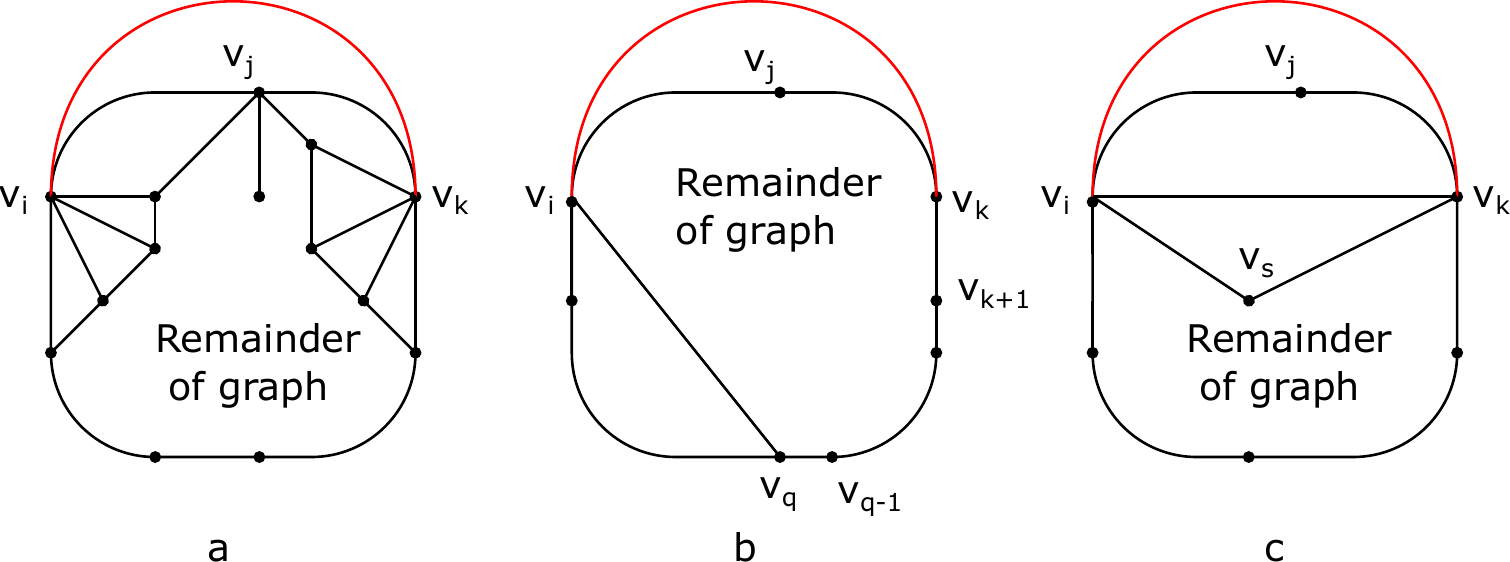}
	\caption{\rm  Three possible depictions of $\mathcal{G}_2$ obtained from $\mathcal{G}_1$ (consists of black edges) by adding a red edge.}
	\label{fig:f5}
\end{figure}
\begin{figure}[H]
	\centering
	\includegraphics[width=1.0\linewidth]{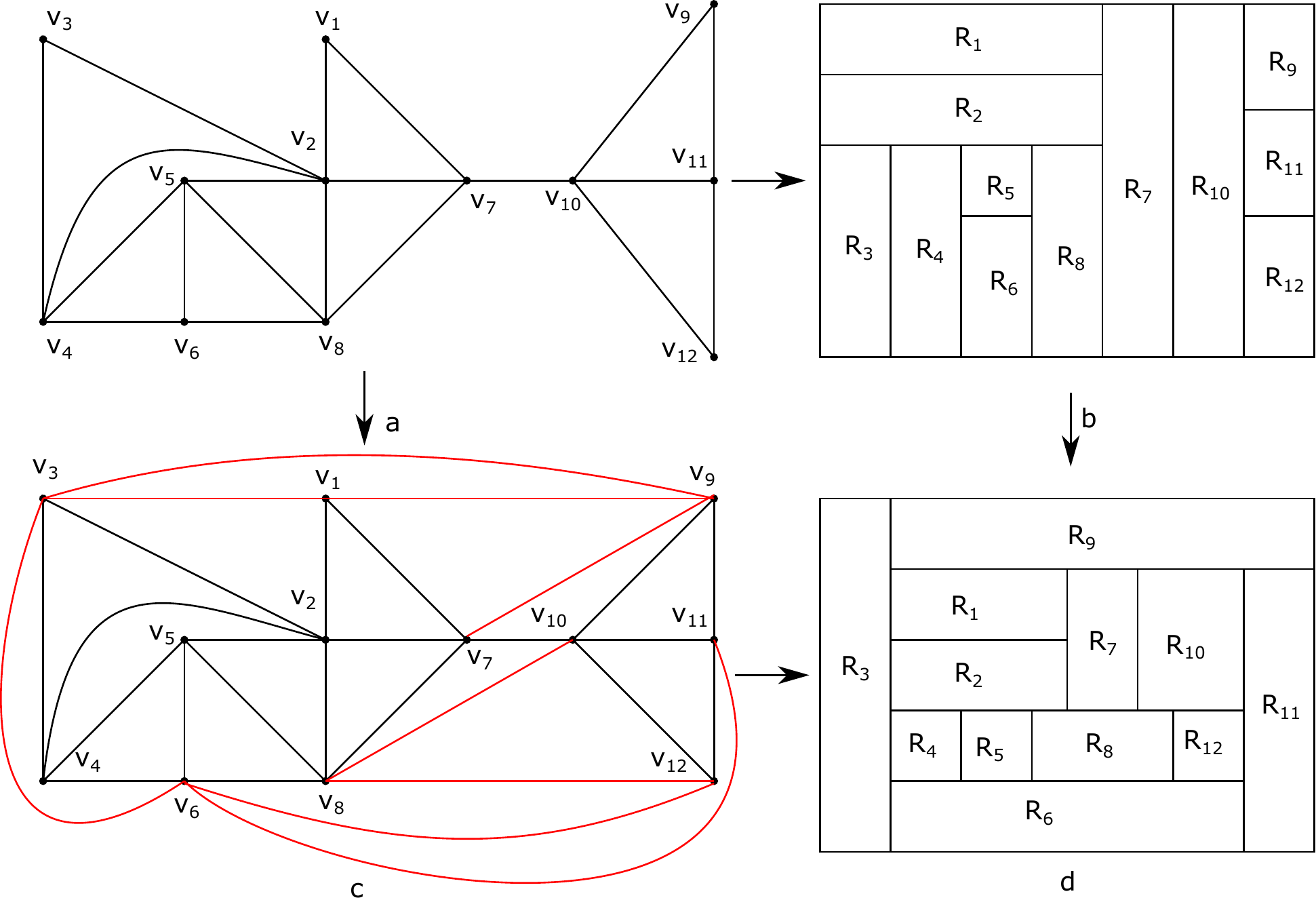}
	\caption{\rm  A given RDG $\mathcal{G}_1$ and (b) its  RFP, (c) the derivation of an MRDG $\mathcal{M}_2$  from $\mathcal{G}_1$, and (d) its RFP.}
	\label{fig:f6}
\end{figure}
  Now it executes the rest of its steps ( $13-20$)  as follows: 
  
  Since for $(v_7,v_{10}) \in L_1$, there is an edge  $(v_{10},v_9)$  belonging to $L_2$, the loop ($13-20$)  adds $(v_7,v_9)$ and $(v_7,v_{10})$ to  $L_2$,  and adds $(v_7,v_9)$ to $E_1$.  Further,  it subtracts $(v_{10},v_9)$ from $L_2$. Thus,  
 $L_2=\{(v_7,v_{10}),$ $(v_7,v_9),$ $(v_1,v_2)$, $(v_2,v_3),$ $(v_3,v_4),$ $(v_4,v_6),$ $(v_6,v_8),$ $(v_7,v_8),$ $(v_{10},v_{12}),$ $(v_{11},v_{12},(v_9,v_{11}),$ $(v_7,v_1)\}$ and  $E_1 =E_1 \cup \{(v_7,v_9)\}$.
 Since $L_1$  has exactly one edge, this loop terminates (In fact, we have transformed the given separable connected RDG to an nonseparable RDG. This method was proved by Lemma \ref{lem44})  and Algorithm \ref{algo41} executes the next loop ($21-28$) as follows:
 \par

 Suppose that Algorithm \ref{algo41} picks   $(v_9,v_7)$ and  $(v_7,v_1)$ from $L_2$.  Since $N(v_1)\cap N(v_9)=\{v_7\}$,  $|N(v_1)\cap N(v_9)|=1$. Then it subtracts $(v_9,v_7)$ and  $(v_7,v_1)$  from $L_2$ and  adds $(v_9,v_1)$  to both $L_2$ and $E_1$ ( the lines 23 and 24). Thus 
 $L_2=\{(v_9,v_1),$ $(v_7,v_{10}),$ $(v_1,v_2)$, $(v_2,v_3),$ $(v_3,v_4),$ $(v_4,v_6),$ $(v_6,v_8),$ $(v_7,v_8),$ $(v_{10},v_{12}),$ $(v_{11},v_{12},(v_9,v_{11})\}$ and $E_1 = E_1 \cup\{(v_9,v_1), (v_7,v_9)\}$. 
 \par
  Again it picks   $(v_{10},v_7)$ and  $(v_7,v_8)$ from $L_2$. Since $N(v_{10})\cap N(v_8)=\{v_7\}$,  $|N(v_{10})\cap N(v_8)|=1$, it subtracts $(v_{10},v_7)$ and  $(v_7,v_8)$ from $L_2$, and adds $(v_{10},v_8)$ to both $L_2$ and $E_1$ (the lines 23 and 24).Thus,
 $L_2=\{(v_{10},v_8),$ $(v_9,v_1),$ $(v_1,v_2)$, $(v_2,v_3),$ $(v_3,v_4),$ $(v_4,v_6),$ $(v_6,v_8),$ $(v_{10},v_{12}),$ $(v_{11},v_{12},(v_9,v_{11})\}$  and $E_1 = E_1\cup \{(v_{10},v_8),$ $(v_9,v_1), (v_7,v_9)\}$.
 \par
  Thus a recursive process of this loop adds  $\{(v_6,v_{11}),$ $(v_3,v_6),$ $(v_9,v_3),$ $(v_9,v_{11}) \}$  to $L_2$ and $E_1 =E_1 \cup  \{(v_9,v_3),$ $(v_6,v_{12}),$ $(v_8,v_{12}),$ $(v_3,v_6),$ $(v_9,v_3),$ $(v_8,v_{10}),$ $(v_9,v_1),$ $(v_3,v_1),$ $(v_9,v_7) \}$ respectively.
 \par
  Since there has not been added any duplicate edge (multiple edges), the loop ($30-35$) skips automatically.
  
Thus we see that the output  is an MRDG $\mathcal{M}_2$ shown in Fig. \ref{fig:f6}c where red edges are the new edges which are added to $\mathcal{G}_1$.  Note that $\mathcal{G}_1$ admits an RFP $\mathcal{F}_1$ shown in \ref{fig:f6}b and $\mathcal{M}_2$ admits an RFP $\mathcal{F}_2$ shown in \ref{fig:f6}d. Consequently to this, $\mathcal{F}_1$ can be  transformed to $\mathcal{F}_2$ (a maximal one).

\subsection*{Analysis of computational complexity}
 Let $v_s$ be a vertex of the largest degree in the input RDG $\mathcal{G}$.  This implies that $|N(v_i| \leq K$ where $d(v_s)=K$. Now we consider each of the following loops: 
	\begin{enumerate}
	\item[i.] The computational complexity of the lines  $3-12$ is $(|N(v_i)||N(v_j)||E_1|) \cong K^2 |E_1|\cong$  O$(n)$. 
	\item[ii.]  The computational complexity of  the lines $14-20$ is $(|L_1|.|L_2|) \cong |L_2| \cong$ O$(n)$. In fact, $L_1$ contains edges whose endpoints are cut vertices. $|L_1| \cong$ O$(n)$ if the given RDG is a path graph. In that case, $L_2$ is empty. Both $L_1$ and $L_2$ can not be large simultaneously. 
	\item[iii.]  The computational complexity of the lines  $21-28$ is $(|N(v_i)||N(v_j)|+|L_2|)|L_2|$ $\cong K^2 |L_2|^2\cong$ O$(n^2)$.
	\item[iv.]  The computational complexity of the lines  $29-35$ is $|E_1|^2 \cong$ O$(n^2)$.
	\end{enumerate}
	 Hence, the computational complexity of Algorithm \ref{algo41} is quadratic.

\begin{remark}
{\rm If $|N(v_i)|$ or $|N(v_j)|$ or $|N(v_i)|\times|N(v_i)|$  is near to $|V|$, then the computational complexity of Algorithm \ref{algo41} becomes O$(n^3)$. However, in design problems such graphs do not appears quite often. Both $|N(v_i)|$ and $|N(v_j)|$ can not be near to $|V|$ in a plane graph.} 
\end{remark}

\section{ Reduction Method} \label{sec5}
In this section, we  show that an edge-reducible biconnected RDG can always be transformed to an edge-irreducible biconnected RDG. Then we present  Algorithm \ref{algo51} that computes the number of CIPs  in a given RDG which is a further input requirement for  Algorithm \ref{algo52} as a call function. Algorithm \ref{algo52} transforms an  edge- reducible biconnected RDG to an  edge-irreducible biconnected RDG.

\begin{theorem} \label{thm51}
	{\rm If an  RDG $\mathcal{G}=(V,E)$ is edge-reducible to an RDG $\mathcal{G}'=(V,E)$ such that $E=E' \cup \{(v_i,v_j)\}$, then $(v_i,v_j)$ is an exterior edge of $\mathcal{G}$.}	
\end{theorem} 

\begin{proof} 
	Assume that  $C$ and $C'$ be the exterior faces of $\mathcal{G}$ and $\mathcal{G}'$ respectively. Since  both $\mathcal{G}$ and $\mathcal{G}'$ are RDGs,   each interior face of both $\mathcal{G}$ and $\mathcal{G}'$ are of equal length ( i.e., of length 3). But $E' \subsetneq E$ and, $\mathcal{G}$ and $\mathcal{G}'$ have the same number of vertices.  This implies that $C$ and $C'$ have different length, i.e.,   $|C|<|C'|$. Also, when   $(v_i,v_j)$ is removed from $C$, the two other edges of the triangle passing through $(v_i,v_j)$ becomes a part of $C'$, i.e., removing an edge from $C$ increases the size of $C'$ by one.     Thus we see that  $|C'|-|C|= 1$ and $(v_i,v_j)$ belongs to $C$. Hence $(v_i,v_j)$ is an exterior edge of $\mathcal{G}$.  
\end{proof}
 
  Theorem \ref{thm51} suggests us that  an edge-reducible RDG can be transformed to any other RDG by deleting  some of its exterior edges. Then further such resultant RDG can also be transformed to another RDG by deleting  some of its exterior edges.  Such recursion process can be continued until the graph remains an RDG. Following this, now we are going to prove  another main result of this paper.

\begin{theorem} \label{thm52}
{\rm  It is always possible to transform an edge-reducible biconnected RDG  to an edge-irreducible biconnected RDG.}
\end{theorem}
\begin{proof}
 Let $\mathcal{G}_1=(V, E_1)$ be an edge-reducible biconnected RDG. Let   $Z_1$  be the set of  all exterior edges of $\mathcal{G}_1$. If for each edge $(v_i,v_j)\in Z_1$,  $|N(v_i)|\leq 2$ or $|N(v_j)|\leq 2$, then $\mathcal{G}_1=(V, E_1)$ is an  edge-irreducible biconnected RDG which is  a contradiction. This implies we can select an edge $(v_i,v_j)$ from $Z_1$ such that $|N(v_i)|> 2, |N(v_j)|>2$. And construct a nonseparable graph  $\mathcal{G}_2=(V, E_2)$  with atmost  4 CIPs  where $E_2=E_1-\{(v_i,v_j)\}$. Such a nonseparable graph always exists otherwise  it has more than four CIPs, $\mathcal{G}_1$ becomes an edge-irreducible RDG which contradicts our assumption.  Since $(v_i,v_j)$ is an exterior edge and $\mathcal{G}_1=(V, E_1)$ is  a biconnected RDG, $|N(v_i) \cap N(v_j)|=1$. Suppose that  $(N(v_i) \cap N(v_j))=\{v_t\}$. 

 Now we claim that $\mathcal{G}_2$ is an RDG. Since $\mathcal{G}_1$ is an  RDG, each of its interior regions are triangular. On removing  an exterior edge from an RDG, the remaining interior regions remains triangular.  This implies that each interior region of $\mathcal{G}_2$ is triangular. It is evident that the removal of an exterior edge from $\mathcal{G}_1$ does not produce a separating triangle. This shows that $\mathcal{G}_2$  is independent of separating triangles. Also, by our assumption, $\mathcal{G}_2$  has atmost four CIPs. By Theorem \ref{thm21}, $\mathcal{G}_2$ is a biconnected RDG. 
 \par
 By continuously defining  $\mathcal{G}_k$ $(k \geq 3)$ as above,   a recursive process shows that at the iteration until the above defined conditions remains true, $\mathcal{G}_k$ is an edge-reducible biconnected RDG. Hence the proof.
 
\end{proof}

    \begin{algorithm}[H] 
	\caption{\textbf{NumberOfCIPs($\mathcal{G}=(V,E),W)$}}
    \begin{algorithmic}[1]
    \REQUIRE  A  biconnected  RDG $\mathcal{G}=(V,E)$
    \ENSURE  Number of  CIPs in $\mathcal{G}$
    
    \STATE  $W\leftarrow \phi$, $L\leftarrow \phi$, $U\leftarrow \phi$, $X\leftarrow \phi$
   \FORALL {$(v_i,v_j)\in E$}	
	\STATE $s \leftarrow |N(v_i) \cap N(v_j)|$
	\IF {$s==1$}
	\STATE $L \leftarrow L \cup \{(v_i,v_j)\}$
	\STATE $U \leftarrow U \cup \{v_i,v_j\}$
	\ELSE
	\CONTINUE
	\ENDIF
	\ENDFOR

	\FORALL { $(v_i,v_j) \in (E-L)$ }
	\IF { $v_i, v_j\in U$ }
	\STATE $W \leftarrow W \cup \{(v_i,v_j)\}$
	\STATE $X \leftarrow X \cup \{v_i,v_j\}$
	\ELSE
	\CONTINUE
	\ENDIF 
	\ENDFOR

	\FORALL { $(v_i,v_j) \in W$ }
	\IF { $(v_i, v_{i+1}),(v_{i+1}, v_{i+2}),\dots,(v_{j-1}, v_j) \in L$}
	\IF { $v_k\in X$, $i+1 \leq k \leq j-1$ }
	\STATE $W \leftarrow W - \{(v_i,v_j)\}$
	\ELSIF {$(v_i, v_{i-1}),(v_{i-1}, v_{i-2}),\dots,(v_{j+1}, v_j) \in L$}
	\IF { $v_k\in X$, $i+1 \leq k \leq j-1$ }
	\STATE $W \leftarrow W - \{(v_i,v_j)\}$
	\ENDIF
	\ENDIF
	\ELSE
	\CONTINUE
	\ENDIF 
	\ENDFOR

	\RETURN $W$
	\end{algorithmic} \label{algo51}
\end{algorithm}

In the most design problems,  graphs structures of floorplans  are biconnected. Therefore abiding by common design practice, we have described Algorithm \ref{algo52} for transforming a biconnected RDG to another biconnected RDG.

\begin{algorithm}[H] 
\caption{\textbf{Restoring an edge-reducible biconnected RDG to an edge-irreducible biconnected RDG}}
\begin{algorithmic}[1]
\REQUIRE  A  biconnected RDG $\mathcal{G}=(V,E)$
\ENSURE  An edge-irreducible biconnected RDG $\mathcal{G}'=(V,E')$
\STATE $Z \leftarrow \phi$

   \FORALL {$(v_i,v_j)\in E$}
   \STATE $s \leftarrow |N(v_i) \cap N(v_j)|$
	\IF {$s==1$}
	\STATE $Z \leftarrow Z \cup \{(v_i,v_j)\}$
	\ELSE
	\CONTINUE
	\ENDIF
	\ENDFOR

   \FORALL {$(v_i,v_j)\in Z$} 
   \IF { $|N(v_i)|> 2 \wedge |N(v_j)|>2 \wedge (N(v_i) \cap N(v_j))=\{v_t\}$}
  \STATE NumberOfCIPs$(\mathcal{G}=(V, E-\{(v_i,v_j)\}),W)$	
   \IF {$|W|\leq 4$}
   \STATE $E \leftarrow E - \{(v_i,v_j)\}$
    \STATE {$Z\leftarrow Z \cup \{(v_i,v_t), (v_t,v_j)\}- \{(v_i,v_j)$}
    \ELSE
    \PRINT $\mathcal{G}$ is an edge-irreducible biconnected RDG. 
    \ENDIF
    \ENDIF
   \ENDFOR
   
   \RETURN $\mathcal{G}$
   \end{algorithmic} \label{algo52}
   \end{algorithm}

For a better understanding of Algorithm \ref{algo52},  we  illustrate its steps through an example.  Consider a biconnected  RDG $\mathcal{G}_1=(V,E_1)$ shown in Fig. \ref{fig:f7}a. First of all, the first loop (the lines $3-9$) of Algorithm \ref{algo52} computes a set $Z$  = $\{ (v_1,v_3)$, $(v_1,v_7),$ $(v_5,v_7),$ $(v_3,v_5)\}$. Now Algorithm \ref{algo52} executes the steps of second loop  (the lines $10-20$) as follows:

 Suppose that the second loop picks  $(v_3,v_5)$ from $Z$ randomly. Then  $11^\text{th}$ line is executed since $N(v_3)=4>2$, $N(v_5)=4>2$ and $N(v_3)\cap N(v_5)=\{v_4\}$. Next it  executes $12^\text{th}$ line to determine a set $W$  of CIPs in $\mathcal{G}_2=(V,E_2)$ where $E_2=E_1-\{(v_3,v_5)\}$. Using Algorithm \ref{algo51},  $|W|=0 \leq 4$. Then $E_1=E_1 -\{(v_3,v_5)\}$ and it   adds both $(v_3,v_4)$ and $(v_4,v_5)$ to $Z$ and remove $(v_3,v_5)$ from $Z$. Thus $Z=\{(v_3,v_4),$ $(v_4,v_5),$  $(v_1,v_3)$, $(v_1,v_7),$ $(v_5,v_7) \}$.
    
 Next suppose that it picks  $(v_1,v_3)$ from $Z$.  Then  $11^\text{th}$ line is executed since $N(v_1)=5>2$, $N(v_3)=4>2$ and $N(v_1)\cap N(v_3)=\{v_2\}$. Next it  executes $12^\text{th}$ line to determine  the set $W$  for $\mathcal{G}_3=(V, E_3)$ where $E_3=E_2-\{(v_1,v_3)\}$. Using Algorithm \ref{algo51}, we have  $W=\{(v_2,v_4)\}$, i.e., $|W|=1 \leq 4$. Therefore it subtracts $(v_1,v_3)$  from both $E_1$ and $Z$, and adds both $(v_1,v_2)$ and $(v_2,v_3)$ to $Z$.  Thus $Z=\{(v_1,v_2),$ $(v_2,v_3),$ $(v_3,v_4),$ $(v_4,v_5),$ $(v_1,v_7),$ $(v_5,v_7)\}$ and $E_1 =  E_1 - \{(v_1,v_3),(v_3,v_5)\}$.

In this way,  the second loop continues until $|W|\leq 4$ with the condition in $11^{\text th}$.  Finally we obtain an edge-irreducible biconnected  RDG shown in Fig. \ref{fig:f7}b with the edge set  $E - \{(v_7,v_9),$ $(v_1,v_9),$ $(v_1,v_7),$ $(v_7,v_5), $ $(v_1,v_3),$ $ (v_3,v_5)\}$ and  $Z$ (set of the exterior edges) becomes $\{(v_7,v_8),$  $(v_8,v_9),$ $(v_1,v_{10}),$  $(v_{10},v_9),$ $(v_5,v_6),$  $(v_6,v_7),$ $(v_1,v_2),$ $ (v_2,v_3),$ $(v_3,v_4),$ $(v_4,v_5) \}$.

 \begin{figure}[H]
	\centering
	\includegraphics[width=0.95\linewidth]{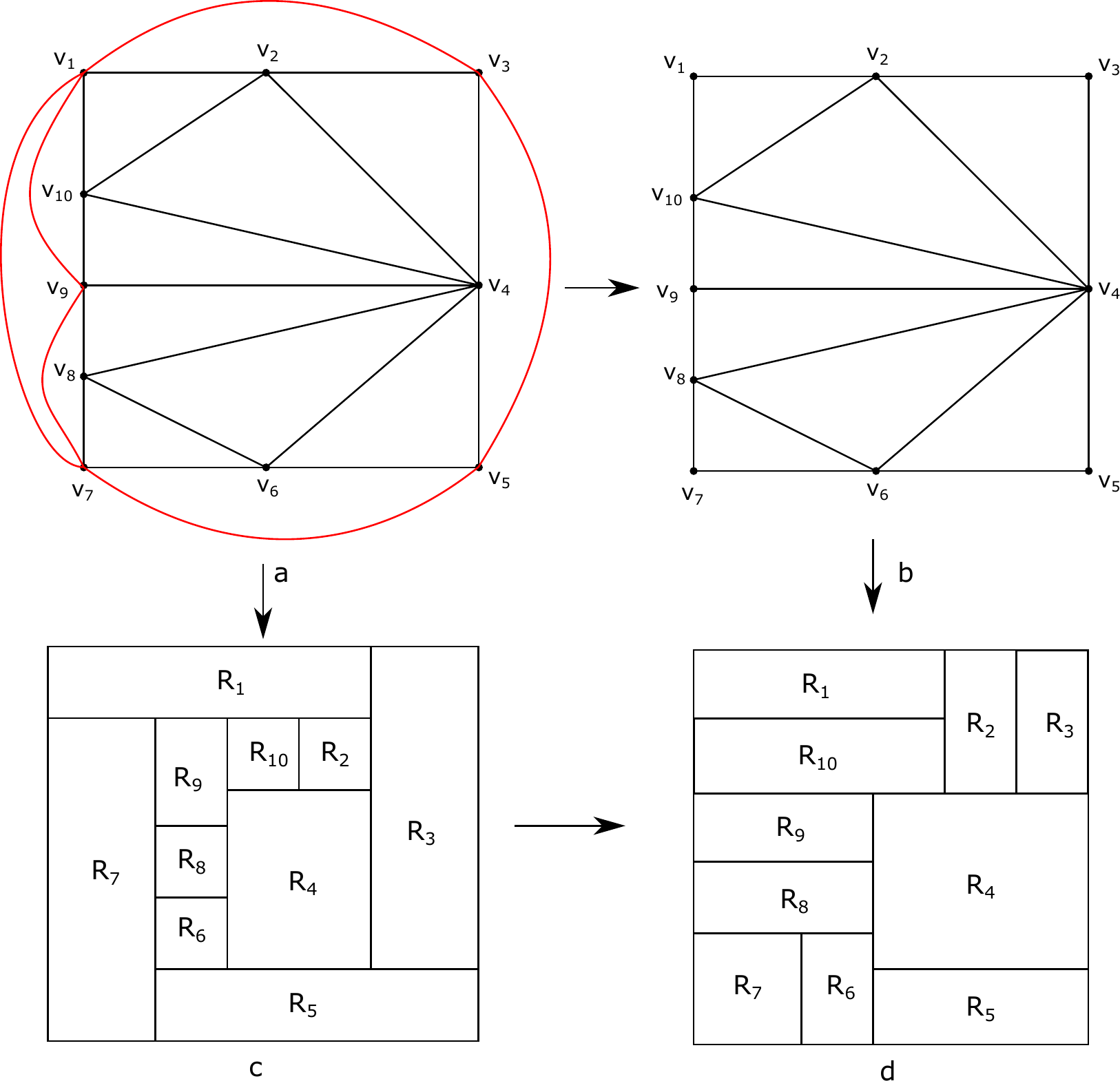}
	\caption{\rm (a) Transformation of an edge-reducible RDG (MRDG) to (b) an edge-irreducible RDG. The respective RFPs are shown in (c) and (d).}
	\label{fig:f7}
\end{figure}
It can be noted that the edge-irreducible biconnected RDG  in Fig. \ref{fig:f7}b can not be further transformed to an edge-irreducible separable connected RDG. In fact, the removal of  a single edge from it violates RDG property (refer to Theorems \ref{thm21} and \ref{thm22}). From this, it follows that a biconnected edge-irreducible RDG may not necessarily be transformed to a separable connected RDG. It may be sometimes possible  to restore it  to an edge-reducible separable connected RDG by relaxing the conditions $N(v_i)>2$ and $N(v_j)>2$ ($11{^\textbf{th}}$ line)\footnote{This condition sustains biconnectedness of the output RDG.}.
Once if it is transformed to  an edge-reducible separable connected RDG, Algorithm \ref{algo52} can be made applicable for separable connected RDGs with a slight modifications  as follows.  

 Consider a separable connected RDG  and construct  its BNG. Theorem \ref{thm22} tells us that  it is a path.  Consider each component corresponding to the vertices of the BNG one by one as an input to Algorithm \ref{algo52}.  To proceed, first consider the components corresponding to the initial and end vertices of the BNG as an input to  Algorithm \ref{algo52}  and follow Theorem \ref{thm22}, i.e., restrict the line 13 of Algorithm \ref{algo52}  by $|W|\leq 2$ such that no CIP should pass through a cut vertex.\footnote{Such CIPs are called critical CIPs.} and restore it to   an edge-irreducible one, it is possible using Algorithm \ref{algo52} since the corresponding block is biconnected.  Then,  the remaining blocks must be  restricted to $|W|=0$ in the line 13 of Algorithm \ref{algo52}. In this way, each biconnected component can be restored to  an edge-irreducible one. Finally  gluing these  edge-irreducible components in the same order as we decomposed them, we obtain an 
 edge-irreducible separable connected RDG from   the  edge-reducible separable connected RDG. 
\par
 The ability of Algorithm \ref{algo52}  in the design process is visible  as follows.  

Algorithm \ref{algo52} gives  an edge-irreducible biconnected  RDG as an output for an input biconnected RDG. But it may not be always preferred/required. Of course, one can obtain output as the  edge-reducible RDG by imposing some restrictions on $Z$ or  $P_c$. Suppose that one desires that a particular set $X$ of adjacency relations must not be removed from the given RDG. Then $Z$ or $P_c$ (in the lines 10 and 13 respectively) needs to be replaced by $Z-X$ or  $P_c-X$. This makes   Algorithm \ref{algo52} more practical to design problems. Consequent to this,  a particular  set of  adjacency relations of component rectangles of an existing RFP  can be sustained by removing  the set $X$ of adjacency relations (edges) of the corresponding vertices in its dual graph as discussed above. Further, Algorithm \ref{algo52} can give also output as  a separable connected RDG for the input biconnected RDG.

 \subsection*{\rm Analysis of computational complexity}

\begin{itemize}

\item The computational complexity of Algorithm \ref{algo51} is linear

The computational complexity of the lines $2-10$ is $|N(v_s)||N(v_t)||E|=K_1K_2|E| \cong$ O$(n)$. The computational complexity of the lines $11-18$ is $|U||E-L| \cong$ O$(n)$. The computational complexity of the lines $19-31$ is $|W||L||X|^2\cong$ O$(n)$. Hence the computational complexity of Algorithm \ref{algo51} is linear.

 \item The computational complexity of  Algorithm \ref{algo52} is O$(n^2)$.

The computational complexity of the lines $3-9$ is  $|N(v_s)||N(v_t)||E|=K_1K_2|E| \cong$ O$(n)$.   The computational complexity of the lines $10-20$ is the product of $|N(v_i)||N(v_j)||Z||P_c||A|$ and the computational complexity of Algorithm \ref{algo51}. But $|N(v_i)||N(v_j)||Z||P_c||A||\cong$ O$(n^2)$. Hence the computational complexity of Algorithm \ref{algo52} is quadratic.
\end{itemize}

\begin{remark}
{\rm If for some graphs, $|N(v_s)||N(v_t)|$ or $|N(v_s)|$  is near to $|V|$, then the computational complexity of both Algorithm \ref{algo51} and Algorithm \ref{algo52} is cubic. Both $|N(v_s)|$ and $|N(v_s)|$ can not  be near to $|V|$ in a plane graph simultaneously. However, in design problem such graphs do not appear quite often.}
\end{remark}

\section{Concluding remarks and future task} \label{sec6}

In this paper, we studied adjacency transformations between RDGs.   We showed  how to transform an RDG into another RDG of which the edge set is a superset or a subset of the first one in quadratic time (the worst case is cubic).

We proved that it is always possible to construct an MRDG from a given RDG, where MRDG represents an RDG with maximum adjacency among given modules.  Then we  presented an algorithm (Algorithm \ref{algo41}) for its construction from the given RDG.  Since adding new edges to an RDG without disturbing RDG property reduce distances among its vertices (usually it is measured by the shortest path between vertices) and hence it is useful in reducing wire-length interconnections among the modules of VLSI floorplans. This method adds new edges to an RDG in bulk if  it is a path graph ( minimal one that is an RDG). In other words, if we pick a Hamiltonian path of an RDG, then a new desired form of the RDG can constructed by adding edges in bulk.  If it is not possible to make  some pair of vertices of a given  RDG adjacent  in its MRDG without disturbing RDG property, then it would be interesting to find a method that can minimize distance between these vertices. In this case, it is equivalent to finding a minimal spanning tree for routability of interconnections.
\par
We also showed  that an edge-reducible RDG  can be restored to a minimal one (an edge-irreducible RDG)  and presented an algorithm (Algorithm \ref{algo52}) to restore the first one to the minimal one.  The removal of an edge from a reducible RDG takes an interior vertex to the exterior. Thus it can be very useful for enhancing the input-output connections between VLSI circuit and the outside world. It would be interesting to derive a necessary and sufficient condition for a given RDG to admit an edge-irreducible RDG.
\par
Consequent to these algorithms, we can construct an efficient RDG by deleting or adding edges to a given RDG.

 \bibliographystyle{elsarticle-num} 
 \bibliography{references}
\end{document}